\documentclass[12pt]{amsart}
\usepackage{latexsym}
\usepackage{graphicx}

\newcommand{\holo}{{\mathcal O}}
\newcommand{\R}{{\mathbb R}}
\newcommand{\C}{{\mathbb C}}

\newcommand{\N}{{\mathbb N}}
\newcommand{\newsection}[1]
{\subsection{#1}\setcounter{theorem}{0} \setcounter{equation}{0}
\par\noindent}

\newtheorem{theorem}{Theorem}

\newtheorem{lemma}[theorem]{Lemma}
\newtheorem{corr}[theorem]{Corollary}

\newtheorem{proposition}[theorem]{Proposition}

\newtheorem{remark}[theorem]{Remark}

\newcommand{\eprop}{\end{proposition}}

\newcommand{\Sum}{\displaystyle\sum}

\newcommand{\haf}{\frac{1}{2}}

\begin{document}
\thispagestyle{empty}
\begin{center}{\bf\large Estimates on the Probability of Outliers for Real Random Bargmann-Fock functions.}\\ Scott Zrebiec \end{center}

\bigskip\bigskip
{\bf Abstract} \par In this paper we consider the distribution of
the zeros of a real random Bargmann-Fock function of one or more
variables. For these random functions we prove estimates for two
types of families of events, both of which are large deviations
from the mean. First, we prove that the probability there are no
zeros in $[-r,r]^m\subset\R^m$ decays at least exponentially in
terms of $r^m$. For this event we also prove a lower bound on the
order of decay, which we do not expect to be sharp. Secondly, we
compute the order of decay for the probability of families of
events where the volume of the complex zero set is either much
larger or much smaller then expected.
\newsection{Introduction}\par
Random functions provide techniques to study typical properties of
elements of a Hilbert space, and have been used to shed light on
the zero set of elements of a Hilbert Space of functions. We will
define real random Bargmann-Fock functions as a linear combination
of basis functions, where each basis function is weighted with a
real standard Gaussian random variable. Alternatively, each random
function is a representative of a Gaussian field on the
Bargmann-Fock space.\par There is a significant body of knowledge
concerning the zero set of Gaussian random functions. In
particular, in the work of Kac and Rice (\cite{Kac43},
\cite{Rice45}) techniques where developed to compute the
correlation functions for the zeros of Gaussian random functions.
This work has more recently been used to compute the two point
correlation function for the zeros of model systems of Gaussian
Random variables, (\cite{BleherDi},
\cite{BleherShiffmanZelditchUniv}, \cite{Prosen96}). Additionally,
there is a significant body of results concerning the expected
number (or volume) of zeros, the variance of the number of zeros
and other results such as a central limit law
(\cite{ShiffmanZelditchVarOfZeros},
\cite{SodinTsirelsonVariance})].
\par To complement these results concerning events within the central limit region,
recent work has been conducted on families of events which are
large deviations from the mean, such as the event where there are
no zeros in an interval (or ball) of large radius. This event is
called a $\mathrm{Hole}$. As we will be working with random
functions on $\C^m$, whose restriction to $\R^m$ is real valued,
we will need to distinguish between real holes and complex ones.

\par{\bf Definition.} We define $\mathrm{Hole}_{r,\R^m} $ to be the event
consisting of all real random Bargmann-Fock functions which have
no zeros in the interval $[-r,r]^m\subset \R^m$, for a large
$r$.\par

\par{\bf Definition.} We define $\mathrm{Hole}_{r,\C^m} $ to be the event
consisting of all real random Bargmann-Fock functions which have
no zeros in the interval $B(0,r)\subset \C^m$, for a large
$r$.\par

A simple family of events similar to those we will study here is
the family of events where a coin is flipped $N$ times but no
heads show up, which has probability $=e^{-N\log(2)}$. To
facilitate the comparison to this toy system we let $N_r$ be the
number of expected zeros (or the expected volume of the zero set)
in a ball of radius $ r$. A series of results have been shown for
the hole probability of complex random functions on various
spaces. As in the coin flip model, a general estimate for the
order of the decay of the upper bound for any Riemann surface was
derived and proven: $\mathbf{ Prob}( \mathrm{ Hole_{r,\C^1}} )\leq
e^{-c N_r}$, \cite{Sodin00}. For one variable complex Gaussian
random functions related to the Hardy space on the disk, the order
of the previous estimate was subsequently shown to be sharp:
$\mathbf{ Prob}( \mathrm{ Hole_{r,\C^1}}) = e^{-\frac{\pi N_r+
o(N_r)}{6}}$, \cite{PeresVirag04}. Whereas, for one complex
Gaussian random function related to the Bargmann-Fock space in $m$
variables the estimate is not sharp
(\cite{SodinTsirelson05},\cite{Zrebiec07MMJ}):
$$e^{-c_2 N_r^{1+\frac{1}{m}}}\leq \mathbf{ Prob}(\mathrm{ Hole_{N_r,\C^1}})\leq  e^{-c_1 N_r^{1+\frac{1}{m}}}.$$
The techniques used in the previous work have also been used to
solve results for random SU($m+1$) polynomials \cite{Zrebiec07},
as well as random holomorphic sections of the N$^{th}$ tensor
power of a positive line bundle on a compact Kahler Manifold,
\cite{ShiffmanZelditchZrebiec}.\par

The study of $\mathrm{Hole_{r,\R^m}}$ poses distinctly different
challenges and opportunities, and significant strides have been
made on one model system: real Gaussian random functions
associated to the $L^1$ Hardy Space. In particular, the
probability a degree $N$ Kac polynomial has no real zeros has been
shown to be $\holo(N^{-b})$, \cite{DemboPoonenShaoZeitouni02}.
Since this work a comparison inequality for Gaussian processes has
been discovered, which is very useful for proving an upper bound
for the hole probability \cite{LiShao02}. This comparison
inequality will prove to be useful in subsequent sections.\par

In this work we will contribute to this area by deriving estimates
for the real and complex hole probability for (real) Gaussian
random functions associated with the Bargmann-Fock Space.

\begin{theorem}\label{RealHPDecay} (The decay of the real hole probability)\par
If
$$\psi_\alpha (z_1, z_2 \ldots, z_m)= \Sum_j \alpha_j
\frac{z_1^{j_1} z_2^{j_2} \ldots z_m^{j_m}}{\sqrt{j_1! \cdot
j_m!}},$$ where $\alpha_j$ are independent identically distributed
real Gaussian random variables, then there exists $R_m, \ c_m, \
C_m$ such that for all $r>R$
$$ e^{-C_mr^{2m}}< \mathbf{Prob}(\mathrm{Hole_{r,\R^m}})< e^{-c_mr^m}$$
\end{theorem}

We expect the order of the upper bound to be sharp and thus the
lower bound to be improvable in general to be $e^{-cr^m}$. It is
worth remarking that while the upper bound in the theorem is the
estimate one would expect based on the coin flip example, the
author initially expected a higher order of decay, based on work
concerning complex zero sets. Computational troubles, followed by
numerical simulations quickly changed the author's expectations.
\par

This result concerning the real zero set contrasts from the
following results for the complex zero sets of real random
Bargmann-Fock functions, which are proven in this article using
the techniques of Sodin and Tsirelson \cite{SodinTsirelson05}.

\begin{theorem}\label{LDofUICF}\label{Hole probability} (Probability estimates for over and under crowded zero
sets)\\
If
$$\psi_\alpha (z_1, z_2 \ldots, z_m)= \Sum_j \alpha_j
\frac{z_1^{j_1} z_2^{j_2} \ldots z_m^{j_m}}{\sqrt{j_1! \cdot
j_m!}},$$ where $\alpha_j$ are independent identically distributed
real Gaussian random variables,
\newline then for all $\delta
> 0,$ there exists $c_1, \ c_{2,\delta}>0 \ and \ R_{m,\delta}$ such that for all $r> R_{m,\delta}$
$$1) \ e^{-c_1r^{2m+2}}<\mathbf{Prob}\left(\left\{\left|n_{\psi_\alpha}(r) -\frac{1}{2}r^2
\right| \geq \delta r^2 \right\} \right) \leq e^{-c_{2,\delta}
r^{2m+2}}$$ where $n_{\psi_\alpha}(r)$ is the unintegrated
counting function for $\psi_\alpha$, and
$$2)\  e^{-c_1 r^{2m+2}} \leq \mathbf{Prob}\left(\mathrm{Hole_{r,\C^m}} \right) \leq
e^{-c_{2, \frac{1}{4}} r^{2m+2}}$$
\end{theorem}

In the above theorem both the upper and lower bound are the same
(orders) one gets when studying complex random Bargmann-Fock
functions in $m$-variables, \cite{Zrebiec07MMJ}.

\par {\bf Acknowledgement:} I would like to thank Bernard Shiffman, Joel Zinn and Mikhail Sodin for many
useful discussions on this subject.

\newsection{Background}\par
Let $m\in \N\backslash \{0\}$ be the number of variables.\par

Throughout this paper we use the standard multi-index notations.
Specifically, if $z\in\C^m, \ z= (z_1, z_2, \ldots, z_m)$ and
$j\in \N^m,\ j= (j_1, j_2, \ldots, j_m)$ then
$$z^j = z_1^{j_1}\cdot z_2^{j_2}\cdot \ldots \cdot z_m^{j_m}$$
$$j! = j_1!\cdot j_2!\cdot \ldots \cdot j_m!$$
$$\zeta \in \C^m, \ z \cdot \zeta = z_1 \zeta_1 + z_2 \zeta_2 +  \ldots +z_m \zeta_m   $$
$$P(0,r)=\{z\in\C^m: |z_1|<r, \ |z_2|<r, \ldots , \ |z_m|<r, \  \} $$
$$B(0,r)=\{z\in\C^m: (\sum |z_j|^2)^\haf <r \} $$

\par

The Bargmann-Fock space is defined to be the set
$$BF=\holo(\C^m)\cap L^2\left(\C^m, \frac{1}{ \pi^m} e^{-|z|^2} dV(z)\right),$$ where
$\holo(\C^m)$ is the set of holomorphic functions on all of $\C^m
$ and $dV(z)$ is Lebesque measure on $\C^m $. With respect to this
norm both
$\displaystyle{\left\{\frac{z^j}{\sqrt{j!}}\right\}_{j\in \N^m}} $
and \newline$\displaystyle{\left\{e^{-\haf y^2+ z\cdot
\overline{y}}\frac{(z-y)^j}{\sqrt{j!}}\right\}_{j\in \N^m}}$ are
orthonormal bases, for all $y \in \C^m$. Further, if $y\in \R^m$
then the restriction of the two previous bases to $\R^m$ is real
valued.
\par

For this paper, we define real random Bargmann-Fock functions as a
linear combination real valued basis elements of the Bargmann-Fock
space, weighted with real independent identically distributed
Gaussian random variables:
$$\psi_\alpha(x)= \sum_{j\in \N^m} \alpha_{j}
\frac{x^j}{\sqrt{j!}}  \eqno{(2)}$$ As $\displaystyle{\lim \sup
|\alpha_j|^{\frac{1}{|j|}}= 1}$ a.s., $\psi_\alpha(x)\in
\holo(\R^m)$ a.s. However, $\psi_\alpha(x)\in (BF)^c$ a.s., since
$\{\alpha_j \}$ is a.s. unbounded. \par

\begin{remark}
In the literature, $\alpha_j$ is often instead taken to be complex
Gaussian random variables, i.e. $\alpha_j = N(0,\frac{1}{\sqrt 2})
+ \sqrt{-1}\cdot N(0,\frac{1}{\sqrt 2})$.
\end{remark}
\par
At first glance the definition of a real random Bargmann-Fock
function seems to depend on the basis chosen but this is not the
case as can be seen in the following two lemmas:

\begin{lemma}\label{Realtranslation}(Real translation invariance law)\par
If $\{\alpha_j\}$ is a sequence of i.i.d. real Gaussian random
variables, then for all $y\in \R^m$ there exists a sequence of
i.i.d. standard real Gaussian random variables $\{ \beta_j \}$
such that for all $x\in \R^m$,
$$\psi_\alpha(x)= e^{\frac{-y^2}{2}+y \cdot x}\psi_\beta(x-y) $$

\end{lemma}

The above lemma illustrates an important tool that we have at our
disposal, but could be reformulated in the terse restatement
``that random Bargmann-Fock functions are well defined,
independent of basis." As stated in the above form, the result
shows that if we know a result for a random Bargmann-Fock function
in one region, we will immediately obtain a similar result on any
other region. A proof of this result may be found in
\cite{XFernique}.
\par If the sequence $\{ \alpha_j\}$ instead was composed of
complex Gaussian random variables the complex Gaussian random
Bargmann-Fock functions would have a complex translation
invariance law (as a Gaussian process). Instead we must be content
with the following:

\begin{proposition}\label{ComplexRotation}(Complex invariance law)\par
If $\{\alpha_j\}$ is a sequence of i.i.d. real Gaussian random
variables, then for all $\zeta \in \C^m,\  \zeta = y\cdot e^{i
\cdot\arg(\zeta)}$, where $y\in \R^m $, there exists a sequence of
i.i.d. standard real Gaussian random variables $\{ \beta_j \}$
such that for all $x\in \R^m$,

$$\psi_\alpha(x)= e^{-\haf|\zeta|^2+x \cdot y}\sum_{|j|=0}^\infty \beta_j e^{-j\cdot i \cdot\arg(\zeta)} \frac{(
xe^{i\arg(\zeta)}-\zeta)^j}{\sqrt{j!}},$$where
$(xe^{i\arg(\zeta)}-\zeta)^j= (x_1 e^{i \arg (\zeta_1)} -
\zeta_1)^{j_1}\cdot\ldots\cdot (x_m e^{i \arg (\zeta_m)} -
\zeta_m)^{j_m}$.
\end{proposition}

\begin{proof}By the previous lemma,\par
\begin{tabular}{rcl}
$\displaystyle{\psi_\alpha(x)}$&${=}$&$\displaystyle{
e^{\frac{-y^2}{2}+xy}\sum
\beta_j \frac{(x-y)^j}{\sqrt{j!}} }$\\
&$=$&$\displaystyle{ e^{\frac{-y^2}{2}+xy}\sum \beta_j e ^{-j\cdot
i \arg(\zeta)} \frac{(e ^{i
\arg(\zeta)}x-\zeta)^j}{\sqrt{j!}}}$\end{tabular}\\
\end{proof}
These two results (Lemma \ref{Realtranslation} and Corrolary
\ref{ComplexRotation}) will be essential in our subsequent work.
Taken together they will allow us to ``translate" any result for
an open neighborhood of the origin to an open set about any other
point.
\par

We need one more technical result:
\begin{proposition}\label{CalcII}
Let $\alpha$ be a standard real Gaussian Random Variable, \\
\begin{tabular}{cl}
then: & a-i) $\mathbf{Prob}(\{  | \alpha| \geq \lambda \}) \leq
\frac{e^{\frac{-\lambda^2}{2}}}{\sqrt{2\pi}}, \ if \ \lambda \geq 1$\\
& a-ii) $\mathbf{Prob}(\{  | \alpha| \leq \lambda \}) \in
\left[\lambda \cdot\frac{ e^{-\haf} }{\sqrt{2\pi}}, \lambda
\cdot\frac{1}{\sqrt{2\pi}}\right], if \lambda \leq 1$
\end{tabular}\par
b) If $\{\alpha_j \}_{j\in \N^n}$ is a set of of independent
identically distributed standard Gaussian random variables, then
$Prob(\{  |\alpha_j| < (1+ \varepsilon)^{|j|}\})=c>0$.\par c) If
$j \in \N^{+,n}$ then $\frac{|j|^{|j|}}{j^{j}} \leq n^{ |j|}$
\end{proposition}\par

\newsection{A Lower bound for Probability of Real hole.}\par
 In this section we prove a lower bound on the order of decay which we do not expect to be
sharp. This bound is however an improvement on the one which could
be derived from the inclusion of $\mathrm{Hole_{r, \C^m}}$ in
$\mathrm{Hole_{r, \R^m}}$.\\
{\bf Theorem \ref{RealHPDecay}} {\it(Lower bound)}\\ {\it If}
$\alpha_j$ {\it is a sequence of real Gaussian random variables
centered about} $\{a_j\}\in \ell^1,$ {\it and} $\psi_\alpha(z) $
{\it is a Gaussian random function associated to the Bargmann-Fock
space then there exist} $C_m, R_m>0$ {\it such that for all}
$r>R_m$
$$\ e^{-c_2 r^{2m}}< \mathbf{Prob}(\{\alpha:\forall x \in [-r,r]^m,
\psi_\alpha(z)\neq 0\})$$
\begin{proof}
As the zero set of a random function associated with the
Bargmann-Fock space is translationally invariant, it suffices to
show, after rescaling,  that there are no zeros in the interval
$(0,r)^m $.\par Let $\Omega_r$ be the event where:\par $i) \
\alpha_{j} \geq E_m +1, \ \forall j: 0\leq \|j\|_{\ell^1}\leq
\lceil 48 m r^2 \rceil= \lceil(m \cdot 2 \cdot 12)(2r)^2 \rceil
$\par
$ii) \ |\alpha_{j}|\leq 2^{\frac{\|j\|_{\ell^1}}{2}},  \ \|j\|_{\ell^1} > \lceil 48 m r^2 \rceil \geq 48 m r^2 $\\
Hence, $\mathbf{Prob}(\Omega_r)\geq C (e^{-c_m r^{2m}})$, by
independence and Proposition \ref{CalcII}.\par



Given $\alpha\in \Omega_r$ and $x\in (0,r)^m$ we now show that $ \alpha$ must belong to $ \mathrm{Hole_{r,\R^m}}$:\\
\begin{tabular}{rcl}
$\psi_\alpha (x)$& $\geq$&
$\displaystyle{\Sum_{\|j\|_{\ell^1}=0}^{\|j\|_{\ell^1} \leq \lceil
24 m r^2 \rceil} \alpha_{j} \frac{x^{\|j\|_{\ell^1}}}{\sqrt{j!
 }} - \Sum_{|j| > \lceil 24 m r^2 \rceil} |\alpha_{j}| \frac{r^{\|j\|_{\ell^1}}}{\sqrt{j!}}}$\\ & $=$& $\displaystyle{\sum^1 -
 \sum^2}$\end{tabular} \\
$\displaystyle{\sum^1\geq \alpha_0\geq E_m +1 }$\\
\begin{tabular}{cl}
 $\Sum^2$ & $\leq \Sum_{|j| > 24 m r^2} 2^{\frac{\|j\|_{\ell^1}}{2}} \left(\frac{\|j\|_{\ell^1}}{24 n}\right)^{\frac{\|j\|_{\ell^1}}{2}} \frac{1}{\sqrt{j!
}}$, as $r<\sqrt{\frac{\|j\|_{\ell^1}}{24 m}} $\\
  & $\leq c \Sum_{\|j\|_{\ell^1} > 24 m r^2} 2^{\frac{\|j\|_{\ell^1}}{2}} \left(\frac{\|j\|_{\ell^1}}{24 m}\right)^{\frac{\|j\|_{\ell^1}}{2}} \prod_{k=1}^{k=m} \left(\frac{e}{j_k}\right)^{\frac{j_k}{2}}$, by Stirling's formula \\
  & $=c\Sum_{\|j\|_{\ell^1} > 24 m r^2} \frac{(\|j\|_{\ell^1})^{\frac{\|j\|_{\ell^1}}{2}}}{\left(\prod_{k=1}^{k=m} j_k^\frac{j_k}{2}\right) m^{\frac{\|j\|_{\ell^1}}{2}}} \left(\frac{e}{12}\right)^{\frac{\|j\|_{\ell^1}}{2}}
  $\\
& $\leq c \Sum_{\|j\|_{\ell^1}>1} \left(\frac{1}{4}\right)^{{\frac{\|j\|_{\ell^1}}{2}}}$\\
& $ \leq c \Sum_{l>1} \left(\frac{1}{2}\right)^{l} l^m\leq E_m $
\end{tabular}\\
Hence, if $\alpha\in \Omega_r$ then for all $x\in[0, r]^m,\
 |\psi_\alpha(x)| \geq  1$\\
\end{proof}

\newsection{An upper bound for the Real hole probability.}\par
Given that $\alpha\in \mathrm{Hole}_{r,\R^m}$ then either the
random function is strictly positive (for $x\in (-r,r)^m$) or
strictly negative. Further, the values at any two distant points
are almost independent. Together, these two observations form the
basis of the intuition behind why one can expect an exponential
decay for the hole probability of a random Bargmann-Fock function,
as at lattice points the values are nearly independent (for a
coarse lattice).\par To make this argument rigorous we will use
the following result of Li and Shao, \cite{LiShao02}. This
comparison also invites comparison with the coin flip model
mentioned in the introduction.

\begin{theorem}\label{NormalComparisonIneq}
let $n\geq 3$ and let $(X_j)_{1\leq j \leq n}$ be a sequence of
standard jointly normal random variables with $E[X_j
X_k]=a_{i,j}$, then

$$2^{-n}\leq \mathbf{Prob}\left(\bigcap_{j=1}^n \{X_j\leq 0\}\right)\leq 2^{-n}
\exp\left\{ \Sum_{1\leq k < j\leq n}\log\left(\frac{\pi}{\pi-2
\arcsin(a_{k,j})}\right)  \right\} $$
\end{theorem}
This lemma is a useful generalization of Slepian's lemma, as it gives both a upper and a lower bound.
In the case of one variable the following result may in fact be proven by using Slepian's lemma,
and comparing discretized real random Bargmann-Fock functions with the Ornstein-Uhlenbeck process.\\
{\bf Theorem \ref{RealHPDecay}} {\it(Upper bound for Theorem
\ref{RealHPDecay})}\\
{\it If} $\alpha_j$ {\it is a sequence of real Gaussian random
variables centered about} $\{a_j\}\in \ell^1,$ {\it and}
$\psi_\alpha(z) $ {\it is a Gaussian random function associated to
the Bargmann-Fock space then there exist} $C_m, R_m>0$ {\it such
that for all} $r>R$
$$ \mathbf{Prob}(\{\alpha:\forall x \in [-r,r]^m\subset \R^m,
\psi_\alpha(z)\neq 0\}) < e^{-c_2 r^{m}}$$

\begin{proof}
Let $Y_t = e^{-\haf x^2}\psi_\alpha(x)$.\par $\displaystyle{ E[Y_t
Y_s]=a_{t,s}=e^{-\haf (s-t)^2}} $\par Let $J=[-r,r]^m\bigcap
(2\N)^m$.
\par For real Gaussian random functions:\par
\begin{tabular}{rcl} $\displaystyle{
\mathrm{Hole_{r,\R^m}}}$&$\displaystyle{ =}$&$\displaystyle{ \{
Y_t
>0, \ t\in[-r, r]^m \} \cup \{ Y_t <0, \ t\in[-r, r]^m \}}.$\\
$ \mathrm{Hole_{r,\R^m}}$& $ \subset$& $ \{ Y_t
>0, \ t\in J \} \cup \{ Y_t <0, \ t\in J  \}.$
\end{tabular}\\
By symmetry $\{ Y_t >0, \ t\in J \}$ and $ \{ Y_t <0, \ t \in J
\}$ have the same probability and it thus
suffice to prove that either of these have exponential decay. By theorem \ref{NormalComparisonIneq} and elementary estimates we may now finish the proof:\\
\begin{tabular}{rcl}
$\displaystyle{\mathbf{Prob}\left(\bigcap_{j\in J} \{Y_j\leq
0\}\right)}$& ${\leq}$& $\displaystyle{ 2^{-|J|} \exp\left\{
\frac{1}{2}\Sum_{ k \neq j, \ k,j \in J}\log\left(\frac{\pi}{\pi-2
\arcsin(a_{k,j})}\right) \right\} }$\\

&${\leq}$& $\displaystyle{ 2^{-|J|} \exp\left\{ \Sum_{j=-\lfloor r
\rfloor}^{ \lfloor r \rfloor} \frac{j^m}{2}
\log\left(\frac{\pi}{\pi-2
\arcsin(e^{-2j^2})}\right) \right\} }$\\

&${\leq}$& $\displaystyle{ 2^{-|J|} \exp\left\{ \Sum_{j=-\lfloor r
\rfloor}^{ \lfloor r \rfloor} \frac{j^m}{2}
\log\left(\frac{\pi}{\pi-3
e^{-2j^2}}\right) \right\} }$\\

&${\leq}$& $\displaystyle{ 2^{-|J|} \exp\left\{ \Sum_{j=-\lfloor r
\rfloor}^{ \lfloor r \rfloor} \frac{j^m}{2} \log\left(1 +
\frac{6}{\pi}
e^{-2j^2}\right) \right\} }$\\

&${\leq}$& $\displaystyle{ 2^{-|J|} \exp\left\{ \Sum_{j=-\lfloor r
\rfloor}^{ \lfloor r \rfloor} \frac{3j^m}{\pi}
e^{\frac{-j^2}{2}} \right\} }$\\
&${\leq}$& $\displaystyle{ 2^{-|J|} C_m = C_m 2^{(2 \lceil r \rceil + 1)^m}}$\\

\end{tabular}
\\
\end{proof}

\newsection{Over and under crowded complex zero sets}

We now restrict ourselves to exploring the event where there are
no complex zeros for a real random Bargmann-Fock function. This is
solved by using techniques developed and used by Sodin and
Tsirelison \cite{SodinTsirelson05} to prove a similar result for
complex random Bargmann-Fock functions in one variable. These
techniques were also generalized to the case of $m$ variable
complex random Bargmann-Fock functions (\cite{Zrebiec07MMJ}) and
may further be generalized to include $m$ variable real random
Bargmann-Fock functions which we do now. The crucial difference
between the argument presented here and that which was presented
in \cite{SodinTsirelson05} and \cite{Zrebiec07MMJ} is that real
random Bargmann-Fock functions are not translationally invariant,
but Lemma \ref{ComplexRotation} will be more than adequate to make
up for this.
\par
We begin this with the following statement about the rate of
growth of a random function.
\begin{lemma}\label{Anchises}
If $\{\alpha_j\}$ is a sequence of independent Gaussian random
variables, and if $\theta_j\in [0, 2\pi]^m$, then for all $\delta
>0$ there exists $c_\delta,\ R_\delta>0$  such that for all $r>R_\delta$,  $$\mathbf{Prob}\left(\left\{ \left|\log
\left(\max_{B(0,r)}\left|\sum \alpha_j e^{\theta_j i}
\frac{z^j}{\sqrt{j!}}\right|\right)- \frac{1}{2}r^2 \right| \geq
\delta r^2\right\}\right) \leq e^{-c_\delta r^{2n+2}} $$
\end{lemma}

Only minor modifications for the argument presented in
\cite{Zrebiec07MMJ} are needed to prove this result in this form.
These modifications are needed as previous versions of this lemma
have only involved complex Gaussian random variables and these
modifications are minor as the steps used in the proof are only
dependent on the norm of the random variables. A sketch of the
proof is included here for completeness sake.

\begin{proof}
Let $\displaystyle{\psi_{\alpha,\theta}= \sum \alpha_j e^{\theta_j
i} \frac{z^j}{\sqrt{j!}} } $\par

Let $\displaystyle{M_{r,\alpha,\theta}= \max_{z\in  B(0,r)}
|\psi_{\alpha,\theta}(z)| } $\par Let $\Gamma_r=\{\alpha:
\frac{\log (M_{r,\alpha,\theta})}{r^2} \geq \frac{1}{2}+ \delta
\}$
\par
The proof that: $\mathbf{Prob}(\Gamma_r) \leq e ^ {- c_{\delta, 1}
r^{2m+2} }$ is extremely easy as $\max_{|j|<N} |\alpha_j|$ is
expected to grow polynomially and would need to grow exponentially
to be in this event. Rigorously,
\\
\begin{tabular}{ccl}
Let $\Omega_r$ be the event where: & $i) \ |\alpha_j|\leq
e^{\frac{\delta r^2}{4}},$ & $ |j| \leq
4 e \cdot m \cdot r^2  $\\
& $ii) |\alpha_{j}|\leq 2^{\frac{|j|}{2}}, $ & $ \ |j| > 4 e \cdot m \cdot r^2$\\
\end{tabular}
\par
$\mathbf{Prob}(\Omega_r^c)\leq e^{-e^{cr^{2}}}\leq
e^{-cr^{2m+2}}$, by Proposition \ref{CalcII}.\par


If $\alpha\in \Omega_r$, then for all $z\in B(0,r)$, then:\\
\begin{tabular}{rll}
$\displaystyle{M_{r,\alpha,\theta}}$&$ \leq$&$\displaystyle{ \max_{z\in B(0,r)}\left(\Sum_{|j| = 0}^{|j|\leq 4 e \cdot m (\frac{1}{2} r^2)} |\alpha_j| \frac{ |z|^{j}}{\sqrt{j!}}+ \Sum_{|j| > 4 e \cdot m (\frac{1}{2}r^2)} |\alpha_j| \frac{ |z|^{j} }{\sqrt{j!}}\right)}$\\
& $ \leq $ &$\displaystyle{e^{(\haf+\frac{\delta}{2})r^2}}$, by a series of standard estimates presented in \cite{Zrebiec07MMJ}.\end{tabular} \\

Thus $\Gamma_r \subset \Omega_r^c $, proving half the result.\par

Let $\displaystyle{M'_{r,\alpha,\theta}=\max_{P(0,r)}|\psi_\alpha|
}$ \par

Let $\displaystyle{ E_{\delta,r,\theta}= \{ \alpha: \log
(M'_{r,\alpha,\theta}< (\haf-\delta)r^2\}\}} $\par If $\alpha\in
 E_{\delta,r,\theta}$ then by Cauchy's Integral Formula: $\left|\frac{\partial^{ j }
\psi_\alpha}{\partial z^j}\right|(0) \leq j! M_{r,\alpha,\theta}'
r^{-|j|}$. Further, by direct computation: $\left|\frac{\partial^{
j } \psi_\alpha}{\partial z^j}\right|(0)= |\alpha_j| \sqrt{j!} $.
When these estimates are combined with elementary estimates and
Stirling's formula, it can be shown that $\exists \Delta>0$ such
that $\forall \delta\leq \Delta$ if $j\in\left[1-\sqrt
\frac{\delta}{n}, 1+\sqrt
\frac{\delta}{n} \right]$ then \\
 $|\alpha_j|\leq e^{-\frac{\delta r^2}{4}}$. For more details see \cite{Zrebiec07MMJ}.
Thus,\\
$\mathbf{Prob}(\{  \log M_{r,\alpha,\theta}' \leq
(\frac{1}{2}-\delta)r^2 \})$
\par$\leq \mathbf{Prob}\left( \left\{ |\alpha_{j}| \leq
e^{-\frac{1}{4}\delta r^2}:\ j_k \ \in \left[(1- \sqrt
\frac{\delta}{m}) r^2, \  (1+ \sqrt \frac{\delta}{m})r^2 \right]
\right\}\right)$\par$ \leq e^{-c_{m,\delta} r^{2m+2}}$, by
Proposition \ref{CalcII}. \par The result then follows by
subadditivity as
$$ M_{r,\alpha, \theta}\geq M'_{\frac{1}{\sqrt{m}} r,\alpha, \theta}\Rightarrow \left\{M_{r,\alpha, \theta}< (\frac{1}{2}-\delta)r^2 \right\}\subset E_{\delta,r\theta}$$

\end{proof}

As a consequence of this previous result and Proposition
\ref{ComplexRotation} (the complex invariance law), we have the
following result:

\begin{corr}\label{ValueEstimate}
For all $\delta>0$ there exists $R_\delta $ such that for all
$r>R_\delta$, if $z_0 \in \overline{B(0,r)}\backslash
B(0,\frac{1}{2}r)$ then
$$ \mathbf{Prob}\left( \left\{ \exists \zeta\in
B(z_0,\delta r) \ s.t. \ \log |\psi_\alpha(\zeta)| >
\left(\frac{1}{2} - \delta \right) |z_0 |^2 \right\}^c \right)\leq
e^{-cr^{2m+2}}$$
\end{corr}
This result has been known for complex random Bargmann-Fock
functions, and the key difference between this proof and its
predecessors is that real random Bargmann-Fock functions are not
``complex translation" invariant, and as such we have to allow for
multiplication of our random variables by $e^{i\theta}, \
\theta\in \R^m$.
\begin{proof}
It suffices to prove the result for any small delta. Let $\delta<
\frac{1}{4}$. Let $ye^{i\theta}=w$, where $y\in\R^m $. Let
$\psi_{\alpha,\theta}(x)=\Sum_{|j|=0}^\infty \alpha_j e^{i \theta
j} \frac{x^j}{\sqrt{j!}}$
\par We restrict ourselves to the following event:
$$ \max_{z \in
\partial B(0,\delta r)}e^{\frac{-\delta^2 r^2}{2}}|\psi_{\alpha,\theta}(z)| \geq -
\delta^3 r^2,$$ whose complement, by Lemma \ref{Anchises}, occurs
with an appropriately small probability.\par Translating this
result using Lemma \ref{ComplexRotation} gives the following:
$$ \max_{z \in B(0,\delta r)}e^{\frac{-(\delta r)^2}{2}}|\psi_{\alpha,\theta}(z)| = \max_{z \in B(0,\delta r)}e^{\frac{-|z|^2}{2}}|\psi_{\alpha,\theta}(z)|
= \max_{z \in B(0,\delta r)}e^{\frac{-|z-y|^2}{2}}|\psi_{\beta}(z-w)|$$\\
\begin{tabular}{rcl}
$\displaystyle{\max_{z \in B(0,\delta
r)}\log(|\psi_{\beta}(z-w)|)}$& $-$& ${\haf|z-y|^2}\
=\displaystyle{\max_{z \in B(w,\delta
r)}\left(\log(|\psi_{\beta}(z)|)-\haf|x|^2\right)}$\\
& $\leq$&$\displaystyle{ \displaystyle{\max_{z \in B(w,\delta
r)}(\log(|\psi_{\beta}(z)|))-\haf|y-\delta r|^2} }$\\
& $\leq$&$\displaystyle{\displaystyle{\max_{z \in B(w,\delta
r)}(\log(|\psi_{\beta}(z)|))-\haf|y|^2+\delta r|y|-\haf \delta^2r^2}} $\\
\end{tabular}\\

Thus,
$$\max_{z\in B(w,\delta r)}\log|\psi_{\beta}(z)|\geq \haf|y|^2+\haf \delta^2 r^2-\delta r|y|-\delta^3 r^2\geq \haf |y|^2- \frac{\delta}{4} |y|^2 $$
\end{proof}

This result allows us to prove the following lemma concerning the
average of $\log |\psi_\alpha |$ with respect to the rotationally
invariant Haar probability measure on the sphere of radius $r$, $d
\sigma_r$:

\begin{lemma}\label{Athena} For all $\delta > 0,$ there exists $c_m>0$ such that for all $r>R_m$
$$ \mathbf{Prob}\left(\left\{ \frac{1}{r^2}
\displaystyle\int_{z\in\partial B(0,r) }  \log |\psi_\alpha | d
\sigma_r(z)  \leq \frac{1}{2} - \Delta \right\} \right) \leq
e^{-cr^{2m+2}}$$
\end{lemma}
Using Corollary \ref{ValueEstimate}, Proposition
\ref{ComplexRotation} and the same steps as those in
\cite{Zrebiec07MMJ} this result may be proven.

\begin{proof} (Sketch).\par
The proof uses regularity properties of subharmonic functions, and
thus we begin by fixing a constant $\kappa_\delta$ which is near
but less than $1$. This is done to avoid singularities of the
Poisson kernel, and still be able to accurately estimate the
average of $\log(|\psi_\alpha|) $. We then choose a disjoint
partition, $\{I_{j,\delta}^{\kappa r} \} $, of $ S_{\kappa r}=
\partial B(0,r)$ by projecting even $2m-1$ cubes so that
$\mathrm{diam}(I^{\kappa r}_j)\leq c_{\delta,m} r$. \par Let
$\sigma_j= \sigma_{\kappa r}(I^{\kappa r}_j)$, which does not
depend on r, and for all $j$ fix a point $x_j \in I_j^{\kappa r}$.

\par By Corollary \ref{ValueEstimate}, for each j there exists a $\zeta_j \in B(x_j,\delta r)$ such that
$$\log (|\psi_\alpha(\zeta_j)|) > \left(\frac{1}{2} - 3 \delta \right)|x_j|^2=
\left(\frac{1}{2} - 3 \delta\right)\kappa^2 r^2 $$ except for N
different events each of which has probability less than
$e^{-c'r^{2m+2}}$, and thus the union of all of these also has
probability less than $e^{-c r^{2m+2}}$.\par As we have the same
estimate for each j for $|\psi_\alpha(\zeta_j)|$, and $\sum
\sigma_j= 1$ we have:
\\
\begin{tabular}{cl}

\hspace{.1in} $\displaystyle{\left(\frac{1}{2}-3 \delta\right)}$ &
$\displaystyle{\kappa^2 r^2 \leq
\Sum_{j=1}^N \sigma_j \log(|\psi_\alpha (\zeta_j)|)}$\\
& $\displaystyle{\leq \int_{\partial B(0,r)}  \left( \Sum_j
\sigma_j P_r(\zeta_j, z) \log(|\psi_\alpha (z)|) d \sigma_r
(z)\right) }$
\\
& $\displaystyle{= \int_{\partial (B(0,r))} \left(\Sum_j \sigma_j
(P_r(\zeta_j, z) -1) \right) \log(|\psi_\alpha(z)|) d \sigma_r (z)}$\\
&$\displaystyle{ \ \ \ \  + \int_{\partial (B(0,r))}
\log(|\psi_\alpha (z)|) d \sigma_r (z)}$
\\
\end{tabular}\\
Hence,
\\\begin{tabular}{rl} $\int_{\partial B(0,r)} $& $ \log(|\psi_\alpha|) d \sigma_r$\\&  $\geq
(\frac{1}{2}-3 \delta) \kappa^2 r^2 -\int |\log|\psi_\alpha|| d
\sigma_r \cdot \max_z |\sum_j \sigma_j (P_r(\zeta_j, z) -1) | $
\\ 
\end{tabular}\\
The result may be completed as it is shown in \cite{Zrebiec07MMJ}
that:
 $$\displaystyle{ \max_{z\in
\partial (B(0,r))} \left|  \sum_j \sigma_j (P_r(\zeta_j,z)
-1 )\right| \leq C_n \delta^{\frac{1}{2 (2n-1)}}},$$ and that
Lemma \ref{Anchises} $\Rightarrow \ \exists c_m, \ R_m,$ such that
for all $r>R_m,$ $$  \newline \mathbf{Prob}\left(\{\|\psi_\alpha
\|_{L^1(S_r,\sigma)}< C_m r^2\}\right)<e^{-c_m r^{2m+2}}.$$
\end{proof}

By combining lemmas \ref{Anchises}, \ref{Athena}, and an
$m$-dimensional analog of Jensen's formula we get
an upperbound for the probability of a large class of events: \\
{\bf Theorem \ref{Hole probability}} {\it(Probability estimates
for over and under crowded zero
sets)}\\
{\it If}
$$\psi_\alpha (z_1, z_2 \ldots, z_m)= \Sum_j \alpha_j
\frac{z_1^{j_1} z_2^{j_2} \ldots z_m^{j_m}}{\sqrt{j_1! \cdot
j_m!}},$$ {\it where} $\alpha_j$ {\it are independent identically
distributed real Gaussian random variables,
\newline then for all} $\delta
> 0,$ {\it there exists} $c_1, \ c_{2,\delta}>0 \ and \ R_{m,\delta}$ such that for all $r> R_{m,\delta}$
$$1) \ e^{-c_1r^{2m+2}}<\mathbf{Prob}\left(\left\{\left|n_{\psi_\alpha}(r) -\frac{1}{2}r^2
\right| \geq \delta r^2 \right\} \right) \leq e^{-c_{2,\delta}
r^{2m+2}}$$ {\it where} $n_{\psi_\alpha}(r)$ {\it is the
unintegrated counting function for} $\psi_\alpha$, {\it and}
$$2)\  e^{-c_1 r^{2m+2}} \leq \mathbf{Prob}\left(\mathrm{Hole_{r,\C^m}} \right) \leq
e^{-c_{2, \frac{1}{4}} r^{2m+2}}$$

\begin{proof} To obtain the upper probability estimate we need only apply
Lemmas \ref{Anchises} and \ref{Athena} to the $m$-dimensional
analog of Jensen's inequality. For further details on this
argument consult the one present in  \cite{Zrebiec07MMJ} which may
be adapted word for word.\par

To finish proving the result it suffices to show that the event
where there is a hole in the ball of radius $r$ contains an event
whose probability is larger than $e^{-cr^{2n+2}}$. \par Let
$\Omega_r$ be the event where:\par
\begin{tabular}{ll}$i)$&$ \ |\alpha_{0}| \geq E_m + 1$,\\ $ii)$&$
\ |\alpha_{j}|\leq e^{-(1+ \frac{m}{2})r^2}, \ \forall j: 1\leq
|j|\leq \lceil 24 m r^2 \rceil= \lceil(n \cdot 2
\cdot 12)r^2 \rceil $\\
$iii)$ & $ \ |\alpha_{j}|\leq 2^{\frac{|j|}{2}},  \ |j| > \lceil
24 m r^2 \rceil \geq 24 m r^2 $\end{tabular}\par $\mathbf{Prob}
(\Omega_r)\geq C (e^{-c_m r^{2m+2}})\geq e^{-cr^{2m+2}}$, by
independence and Proposition \ref{CalcII}.\par If $\alpha \in \Omega_r$ then\\
$|\psi_\alpha(z)| \geq |\alpha_{0}| - \Sum_{|j|=1}^{|j| \leq
\lceil 24 m r^2 \rceil} |\alpha_{j}| \frac{r^{|j|}}{\sqrt{j!
 }} - \Sum_{|j| > \lceil 24 m r^2 \rceil} |\alpha_{j}| \frac{r^{|j|}}{\sqrt{j!}}= |\alpha_{0}|- \sum^1 - \sum^2$\\

\begin{tabular}{cl}
 $\Sum^1$ & $\leq e^{-(1+ \frac{m}{2})r^2} \Sum_{|j| = 1}^{|j| \leq \lceil 24 m r^2 \rceil} \frac{r^{|j|}}{\sqrt{j!}}
 $ \\
 \vspace{.1in}
 &$\leq e^{-(1+ \frac{m}{2})r^2} \sqrt{(24 n r^2 +1)^{m}} \sqrt{(e^{r^m})}$,
 by Cauchy-Schwarz inequality.
\\
 \vspace{.1in}
 &$\leq C_m r^{m} e^{-r^2}\leq c e^{-0.9 r^2}<\frac{1}{2} $ for $r> R_m$\\
\end{tabular}\\

\begin{tabular}{cl}
 $\Sum^2$ & $\leq \Sum_{|j| > 24 m r^2} 2^{\frac{|j|}{2}} \left(\frac{|j|}{24 m}\right)^{\frac{|j|}{2}} \frac{1}{\sqrt{j!
}}$, as $r<\sqrt{\frac{|j|}{24 m}} $\\
  & $\leq c \Sum_{|j| > 24 m r^2} 2^{\frac{|j|}{2}} \left(\frac{|j|}{24 m}\right)^{\frac{|j|}{2}} \prod_{k=1}^{k=m} \left(\frac{e}{j_k}\right)^{\frac{j_k}{2}}$, by Stirling's formula \\
  & $=c\Sum_{|j| > 24 m r^2} \frac{(|j|)^{\frac{|j|}{2}}}{\left(\prod_{k=1}^{k=m} j_k^\frac{j_k}{2}\right) m^{\frac{|j|}{2}}} \left(\frac{e}{12}\right)^{\frac{|j|}{2}}
  $\\
& $\leq c \Sum_{|j|>1} \left(\frac{1}{4}\right)^{{\frac{|j|}{2}}}$, by Proposition \ref{CalcII}.\\
& $ \leq E_m $
\end{tabular}\\
Hence, $|\psi_\alpha(z)| \geq E_m + 1- \Sum^1-\Sum^2\geq \frac
{1}{2} $
\end{proof}


\end{document}